\documentclass[EJP,preprint]{ejpecp} 

\usepackage[utf8]{inputenc}
\usepackage{natbib}
\usepackage[format=plain, font=it, labelfont=it]{caption}
\usepackage{enumitem}

\SHORTTITLE{Copy number variation of genetic elements}

\TITLE{A diploid population model for copy number  variation \\ of genetic elements}

\AUTHORS{Peter Pfaffelhuber\footnote{University of Freiburg, Germany.    \EMAIL{p.p@stochastik.uni-freiburg.de}} \and Anton Wakolbinger\footnote{Goethe University Frankfurt, Germany. \EMAIL{wakolbin@math.uni-frankfurt.de}}}
\KEYWORDS{Poisson approximation; Feller branching diffusion; transposable elements; slow-fast system} \AMSSUBJ{92D15} 
\AMSSUBJSECONDARY{60J80; 60F17; 60G57}
\SUBMITTED{April 25, 2022} 
\ACCEPTED{???} 

\VOLUME{0}
\YEAR{2022}
\PAPERNUM{0}
\DOI{10.1214/YY-TN}

\ABSTRACT{We study the following model for a diploid population of constant size $N$: Every individual carries a random number of (genetic) elements. Upon a reproduction event each of the two parents passes each element independently with probability $\tfrac 12$ on to the offspring. We study the process $X^N = (X^N(1), X^N(2),...)$, where $X_t^N(k)$ is the frequency of individuals at time $t$ that carry $k$ elements, and prove convergence (in some weak sense) of $X^N$ jointly with its empirical first moment $Z^N$ to the ``slow-fast'' system $(Z,X)$,  where $X_t = \text{Poi}(Z_t)$ and $Z$ evolves according to a critical Feller branching process.  We discuss heuristics explaining this finding and some extensions and limitations.}

\makeatletter
\renewcommand{\@fnsymbol}[1]{\@arabic{#1 }}
\makeatother

\begin{document}
\section{Introduction and main result}\label{intro}
The motivation of the present study is twofold. From a mathematical point of view it leads to a new large population limit of a system of particles performing a coordinated and spatially structured critical branching in a rapidly fluctuating random environment given by a bi-parental Moran graph; see  also Remark \ref{rem:heuristics} for a more detailed description. The biological motivation is to model the  evolution of transposable elements. These are repetitive sequences of 100 base pairs or longer, which are able to relocate within the genome of a host; see e.g.\ \cite{pmid30454069} for a review on transposable elements. 

We thus consider a population of size $N$ (the number of diploids) undergoing random reproduction events at rate $N^2/2$. Each individual has a type in $\mathbb N_0$, where type $k$ means that the individual carries $k$ genetic elements (GEs) in its genome. At a reproduction event a randomly chosen individual dies, and a randomly chosen pair of individuals produces some offspring. If the types of the parents are $k$ and $l$, then the type of the offspring has  a binomial distribution with parameters $k+l$ and $\tfrac 12$, i.e.\ it inherits each parental GE independently with probability $\tfrac 12$. Thus, if $x_k$ denotes the current fraction of individuals of type $k \in \mathbb N_0$, and $x:= (x_k)_{k\in \mathbb N_0}$, then  the jump rate from  \mbox{$x$ to  $x + (e_m - e_n)/N $ is}
\begin{align}\label{eq:jumpRate}
  \tfrac {N^2}2 x_n \sum_{k,l} x_k x_l \binom{k+l}{m} 2^{-(k+l)},
\end{align}
where $e_m$ is the $m$th unit vector, $m=0,1,2,...$. 
This gives rise to a Markovian jump process $X^N := (X^N_t)_{t\ge 0}$ taking its values in $\mathcal P(\mathbb N_0)$, the set of probability measures on~$\mathbb N_0$ endowed with the topology of weak convergence. We write $$Z^N_t := \sum_{k=1}^\infty k X^N_t(k)$$ for the average number of GEs per individual at time $t$. Our main result concerns the convergence in distribution  of the sequence of stochastic processes $\left(Z^N, X^N\right)$ as $N\to \infty$. The limiting process will turn out to be $(Z,X)$, where $Z$ is a {\em standard Feller branching diffusion} obeying the SDE
\begin{equation} \label{FBD}
dZ_t = \sqrt{Z_t}\, dW_t, \qquad Z_0 =z
\end{equation}
and 
\begin{equation} \label{Poi}
X_t = {\rm Poi}(Z_t),   \quad t > 0,
\end{equation}
where  ${\rm Poi}(\lambda)$ denotes the Poisson distribution on $\mathbb N_0$ with parameter $\lambda > 0$.

To specify a topology underlying this convergence we define the  {\em weighted occupation measure} of $\xi \in \mathcal D(\mathcal P(\mathbb N_0))$, the set of c\`adl\`ag $\mathcal P(\mathbb N_0)$-valued paths indexed by $t \in [0,\infty)$, 
as the  probability measure
\begin{align}\label{occmeas}
\Gamma_\xi([0,t] \times A) := \int_0^t e^{-s} \mathbf 1_{\{\xi_s\in A\}} ds,
\end{align}
where $t \ge 0$ and $A$ is a measurable subset of $\mathcal P(\mathbb N_0)$. Following \cite{Kurtz1991} we say that a sequence $(\xi^N)$ in $\mathcal D(\mathcal P(\mathbb N_0))$ {\em converges  in measure} to a $\xi \in \mathcal D(\mathcal P(\mathbb N_0))$ if the sequence of probability measures  $\Gamma_{\xi^N}$ converges weakly to $\Gamma_{\xi}$. On $\mathcal D(\mathbb R_+)$ we will use the Skorokhod topology, see e.g.\ Chapter~3 in \cite{EthierKurtz1986}.
\begin{theorem}\label{T1}
  Let $X^N$ be the $\mathcal P(\mathbb N_0)$-valued Markov jump process with jump rates as in \eqref{eq:jumpRate}, starting in $X^N_0$ with atoms of size $N^{-1}$. Assume that, for some $z > 0$, 
  $Z^N_0 \xrightarrow{N\to\infty} z$ in probability, and $\sup_N  \mathbb E\Big[\sum_{k=1}^\infty k^3 X^N_0(k)\Big] < \infty$. Then $ (Z^N, X^N)$ converges in distribution to $(Z,X)$ obeying \eqref{FBD} and \eqref{Poi}, where  $\mathcal D(\mathbb R_+) \times \mathcal D(\mathcal P(\mathbb N_0))$ is equipped with the Skorokhod topology in the first, and with the topology of convergence in measure in the second coordinate.
\end{theorem}
In particular, Theorem~\ref{T1} shows that the average number of GEs per individual becomes Markovian in the limit $N\to \infty$. This average follows the dynamics of a Feller branching diffusion, and the distribution of the total number of GEs is Poisson at  all points in time. The proof of Theorem~\ref{T1} can be found in Section~\ref{Secproof}.

\section{Perspectives and background}\label{secPB}
\begin{remark}[An individual-based graphical construction]\label{rem:heuristics}
  The following  individ\-ual-based construction of the process $X^N$ gives a heuristic explanation of why Poisson limits and Feller's branching diffusion  appear in the situation of Theorem \ref{T1}.  Let $\Pi_{hij}$,  $h,i,j \in \{1,...,N\}$, be a family of independent rate $N^{-1}$ Poisson point processes on the time axis. At each time point $t$ of $\Pi_{hij}$ draw a pair of arrows, one from $(i,t)$ to $(h,t)$ and one from $(j,t)$ to $(h,t)$. This gives rise to the  {\em bi-parental Moran graph} $\mathcal G_N$ with vertex set $\{1,...,N\} \times \mathbb R_+$; see Figure \ref{Moran2}. 
  
\begin{figure}[ht]
  \begin{center}
    \includegraphics[width=0.8\textwidth]{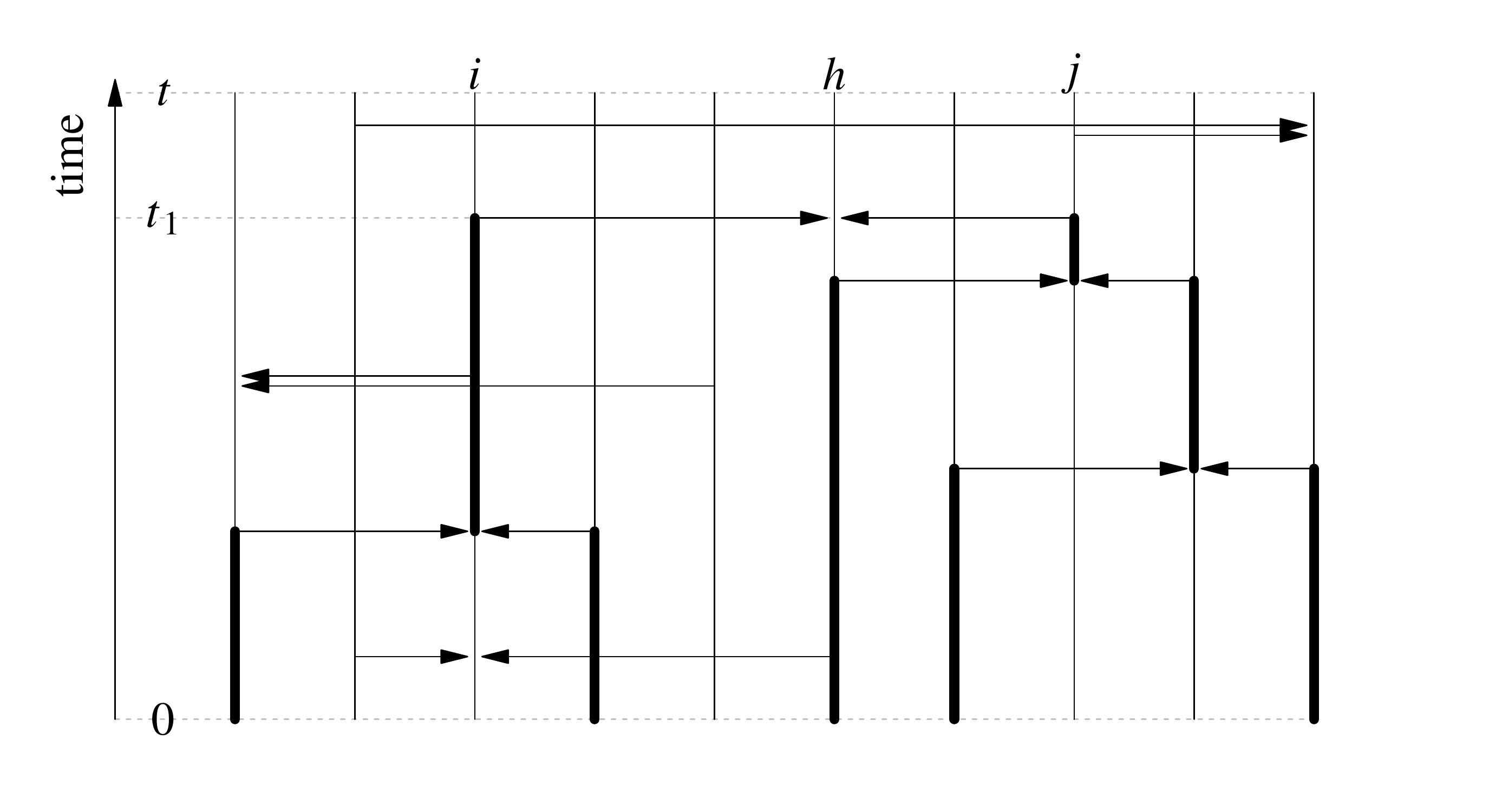}
  \caption{A detail of the bi-parental Moran graph $\mathcal G_N$, with a point of $\Pi_{hij}$ at $t_1$. In this realisation the ancestry of the individual $(h,t_1)$ (drawn with bold lines) is a binary splitting tree $\mathcal T$  with root at  $(h,t_1)$. 
}
\label{Moran2}
\end{center}
\end{figure}

 The graph~$\mathcal G_N$ serves as a random environment  for a coordinated, structured branching process of the population of GEs. Specifically, for $t_1 \in \Pi_{hij}$ (as in Figure \ref{Moran2}),  each GE arriving at  $(i,t_1-)$   tosses a fair coin. In case of ``success'' it puts one offspring at $(h,t_1)$, and in any case it continues to live at $ (i,t_1)$ if $h\neq i$. The same happens for the GEs arriving at $(j,t_1-)$. The population at $(h,t_1-)$ is replaced by the sum of the offspring of the populations at $(i,t_1-)$ and $(j,t_1-)$. 
  
  Now consider the genealogy of the GEs in an {\em annealed} picture, i.e. averaged over $\mathcal G_N$. Then the offspring of a single GE experiences a critical binary branching process with branching rate~$N$ and with ``locally coordinated branching''  in the sense that GEs living in the same host are affected by simultaneous reproduction events. This local coordination of the branching  induces a dependence also between the offspring of different GEs that were present at time $t=0$. It turns out, however, that the population of GEs  is continuously spread quickly enough over the space $\{1,...N\}$ of hosts so that the dependencies introduced by the local coordination become negligible as $N\to \infty$. 
  
  More precisely, let the GE numbers at time $t=0$  be given by a, say,  i.i.d. family $(\zeta_i^N(0))_{i\in [N]}$ of $\mathbb N_0$-valued random variables with finite third moment, and let  $(\zeta_i^N(t))_{i=1,...,N}$ be the GE numbers at time $t$ that arise through the reproduction dynamics along the random graph $\mathcal G_N$ as described above. Then already at time $t_1= \log \log N / N$ (and thus before an effective change of the total number of GEs has occurred) the numbers $\zeta^N_h(t_1)$ of GEs in host $(h,t_t)$ are close to Poisson. 
  
  This can be seen by tracing the ancestry of $(h,t_1)$ in $\mathcal G_N$ back to time $0$; see the bold lines giving rise to a binary tree in Figure~\ref{Moran2}. Thanks to the assumption $t_1=\log \log N / N$, the parental lineages of  $(h,t_1)$ from time $t_1$ down to time $0$ collide only with small probability; hence the  ancestry of $(h,t_1)$ forms with high probability a binary splitting tree~$\mathcal T$ with root at $(h,t_1)$, (order of) $\log \log N$ generations and (order of) $\log N$ leaves at time~0. In the situation of Figure \ref{Moran2} the random number $\zeta^N_h(t_1)$ arises, conditional on $\zeta^N_i(t_1)$ and $\zeta^N_j(t_1)$, as the sum of two independent binomially distributed random variables  with parameters $\zeta^N_i(t_1), \frac 12$ and $\zeta^N_j(t_1), \frac 12$, respectively. Playing this back to time 0 along the tree $\mathcal T$, a reasoning similar as in the proof of Lemma \ref{l:4}  shows that the distribution of $\zeta^N_h(t_1)$ is close to Poisson for large $N$.  
  
  It is easily seen that  the total number of GEs, $NZ^N_t = \sum_{i=1}^N \zeta^N_i(t)$, is a martingale. As time proceeds, the ongoing (quick) Poissonization of $(\zeta_i^N)_{i=1,...,N}$ happens conditional on the current value of $Z^N$.
As it turns out (cf. Proposition \ref{P:2}), the near-Poissonicity of $(\zeta_i^N(t))_{i=1,...,N}$ helps to control the quadratic variation of $Z^N$ and to prove that $Z^N$ converges as $N\to \infty$ to a standard Feller branching diffusion.
\end{remark}
\begin{remark}[Stochastic slow/fast systems]\label{rem:slowfast}
As we will see, $(Z^N, X^N)$ is a slow/fast system, for which $\mathcal{POI}$, the set of Poisson distributions on $\mathbb N_0$, forms a stable manifold. Such systems have been studied intensively, see e.g.\ \cite{Kurtz1992, PardouxVere1, berglund2006noise}. We could not find a result in the general theory which covers the situation of our Theorem, but the method of \cite{Katzenberger1991} comes pretty close. More precisely, as will become clear in the proof of Theorem~\ref{T1}, $Z^N$ is slow in the sense that there is an operator which describes the (asymptotically) fastest part of the dynamics of $X^N$ and which vanishes on functions depending only on the first moment of $x$; see Lemma~\ref{l1} and \eqref{eq:POIrho2}. In other words, the fast dynamics has the property that $\mathcal{POI}$ is invariant. The dynamics is therefore only governed by motions within $\mathcal{POI}$; this as well as the convergence of  $Z^N$ to  $Z$ is guaranteed by Theorem~\ref{T1}. 

The setting of a slow/fast dynamics giving rise to a dynamics on a lower dimensional manifold was given (for finite-dimensional semimartingales) in \cite{Katzenberger1991}. We failed to show the conditions of  \cite{Katzenberger1991} on the convergence towards the manifold, since we are lacking a general bound how the fast dynamics pushes the system to the manifold.\footnote{We will see in the proof of Corollary \ref{cor1} that  the difference between the second factorial moment and the square of the first moment of $X_t^N$ converges  exponentially fast to $0$. With additional efforts, and under suitable assumptions on the initial condition,  it might be possible  to obtain by similar techniques an analogous result also for all the higher factorial moments. However, it remains unclear whether such a result would help to meet the conditions of Theorem 5.1 of \cite{Katzenberger1991} (when extended to the infinite-dimensional setting).} Rather, we use martingale arguments for this convergence; see Proposition~\ref{P:3}. These are not sufficient to show convergence in the Skorokhod sense (as in \cite{Katzenberger1991}), but only in the sense of convergence of occupation measures. Thus, the question whether $X^N$ converges to $X$ in the Skorokhod sense on each time interval $[\varepsilon, \infty)$ remains open.

\end{remark}

\begin{remark}[Convergence in measure] \label{Remtop}
Tightness criteria for c\`adl\`ag processes with respect to convergence in measure were given in \cite{MeyerZheng1984}, and refined in \cite{Kurtz1991}. In contrast to these approaches, we show convergence in measure of $X^N$ in two steps. First, we show tightness of $\Gamma_{X^N}$ (as random probability measures on $[0,\infty) \times \mathcal P(\mathbb N_0))$. In a second step, we show that any limit $\Gamma$ must a.s. be concentrated on $[0,\infty)\times \mathcal{POI}$. We then prove convergence of $Z^N$ to the Feller branching diffusion $Z$ and conclude that $\Gamma = \Gamma_X$ for $X = ({\rm Poi}(Z_t))_{t\geq 0}$; see Section~\ref{ss:35}.
\end{remark}

\begin{remark}[Moment assumptions] \label{MomAss}
The assumption of uniform boundedness of the expectation of the third moment of $X_0^N$  translates via Corollary \ref{cor1} to $X_t^N$. This is used in the proof of the key Proposition \ref{P:3} to obtain the uniform integrability of the second moment of $Y^N$, where $Y^N$ is distributed according to the weighted occupation measure of $X^N$. One may conjecture that the uniform integrability of the second moments of $X_0^N$  is enough to guarantee the uniform integrability of $\rho_2(Y^N)$. Proving this, however, seems to require considerable additional technical efforts.
\end{remark}
\begin{remark}[Transposable elements]
  Let us briefly explain the assumptions on the dynamics of $X^N$ as a model for the evolution of transposable (genetic) elements (GEs). In the  model that underlies \eqref{eq:jumpRate}, we implicitly assume that each individual is diploid in the sense that it consists of two sets of chromosomes. The assumption that the GEs of both parents are inherited independenly is satisfied when GEs jump frequently within the genome, and GEs on each chromosome are not inherited together, which is the case in a high-recombination limit (also referred to as free recombination).
  Under these assumptions, whenever an individual carrying a total of $k$ GEs produces an offspring, it passes on one copy of its set of chromosomes, which carries a 
  binomially distributed number of GEs with parameters $k,\tfrac 12$. Since the offspring has two parents with $k$ and $l$ GEs, it thus has a binomially distributed number of GEs with parameters $k+l,\tfrac 12$. From this explanation we find the two main assumption for the biological dynamics on diploids:
  \begin{itemize}
      \item GEs jump frequently;
      \item Recombination between GEs is free.
  \end{itemize}
\end{remark}

\begin{remark}[Context and novelty of the model]
Bi-parental population models have been studied to some extent. \cite{Chang1999} and \cite{pmid15457259} analyse the common ancestry of all living humans. \cite{Coron2020GeneticsOT} study the distribution of the genetic material which an ancestor contributes to today's population. They consider a scaling limit of a biparental model in which first the number of generations and then the population size tends to infinity. Other population genetic models (e.g.\ \citealp{lambert2021chromosome}) implicitly assume two parents by using an ancestral recombination graph \citep{GriffithsMarjoram1997}. Wakeley and co-authors study the effect of the bi-parental pedigree on the evolution of allele frequencies (e.g.\ \citealp{   pmid22234858, pmid27432946}). Additional related work has been done by  \cite{baird2003distribution} (for a study of the amount of genetic material from a single individual which is present anywhere in the population after some time) and \cite{foutel2022convergence} (for a branching process approximation within the same model).

A novelty in our model lies in the study of a large population on the evolutionary timescale (i.e.\ one unit of time is $O(N)$ generations) together with a free recombination on the generation timescale, in the sense that the two parents of each individual are effectively involved at each reproduction event.  Thus, in our setting we are able to combine  the rescaling of time with a rapid {\em and} free recombination, and to prove convergence in the limit of large populations on the evolutionary timescale.
In spirit, our model fits the framework of the Poisson Random Field approach taken by \cite{SawyerHartl1992}; see also \cite{PRFtutorial2008} for a review. In such models, all loci evolve independently due to free recombination. We do not model genomic loci explicitly  but consider the total number of GEs, which are distributed somewhere in the genome. Starting from an individual-based finite poplation model, we show that this total number is for large populations asymptotically Poisson in each individual genome, with a random intensity that follows Feller's branching diffusion on the evolutionaly timescale.
\end{remark}

\begin{remark}[Extensions]
  The population model with jump rates given through \eqref{eq:jumpRate} is neutral in the sense that (i) the number of GEs does not change on average in all individuals and (ii) the probability to be involved in a reproduction event does not depend on the number of GEs an individual carries. Both assumption can be relaxed. For (i), we might assume that an individual of type $k$ acquires new GEs at rate $\mu + k\nu$, and each GE is lost (or silenced) at rate $\beta$. For (ii), we might assume that an individual of type $k$  is chosen as a parent with probability proportional to $(1-\alpha/N)^k \approx 1- k\alpha/N$, for some $\alpha \in \mathbb R$. Under this selective model, we strongly conjecture that Theorem~\ref{T1} still holds, but with \eqref{FBD} changed to
  
  \vspace{-0.3cm}
  $$ dZ = (\mu + (\nu - \beta - \alpha)Z) dt + \sqrt{Z}dW,$$ 
  i.e.\ to a non-critical Feller branching diffusion with immigration.
\end{remark}

\section{Proof of Theorem~\ref{T1}}\label{Secproof}
The arguments from the graphical construction in Remark~\ref{rem:heuristics} give some intuition why Theorem~\ref{T1} holds.
However, our proof proceeds via a different route. The main steps are as follows: After introducing some notation in Section~\ref{ss:31}, we analyse the generator of $X^N$ in Section~\ref{ss:32} by collecting terms which are of order $N^1, N^0, N^{-1},...$, i.e.\ $G^N = NG_1 + G_0 + O(N^{-1})$; see Lemma~\ref{l2}. Moreover, we will see that $G_1 f = 0$ if $f$ only depends on $x$ via its first moment $\rho_1(x)$; see Lemma~\ref{l1}. For convergence to the manifold of Poisson distributions
we require only the highest order term, $G_1$. In Section~\ref{ss:33}, we first give a characterization of random Poisson distribution in Lemma~\ref{l:4} and use this in order to show that the limit of occupation measures $\Gamma_{X^N}$ is concentrated on Poisson distributions using  martingale arguments; see Proposition~\ref{P:3}. Then, in Section~\ref{ss:34}, we use these results in order to show convergence of $Z^N$ as $N\to \infty$; see Proposition~\ref{P:2}. Finally, in Section~\ref{ss:35} we collect these insights to complete the proof of Theorem~\ref{T1}. 

\subsection{Notation and basics}\label{ss:31} 
\begin{definition}[Moments, generating functions, state space]
\label{def:1}
\begin{enumerate}[wide, labelindent=0pt]
\item \sloppy We identify a probability measure $x$ on $\mathbb N_0$ with the sequence $(x_0, x_1,...)$ of its weights. For $j=1,2,...$, we denote by 
    $$  \rho_j(x) := \sum_{k=j}^\infty k \cdots (k-j+1) x_k$$
the $j$th factorial moment of $x$. In particular, $\rho_1(x)$ is the mean of $x$. We put
    \begin{align}\label{eq:psis} \psi_s(x) := \sum_{k=0}^\infty x_k(1-s)^k, \qquad s\in[0,1].
    \end{align}
    \item The state space of $X^N$ is
    \begin{align}\label{eq:E}
    E_N := \Big\{x \in \mathcal P(\mathbb N_0): N x_k  \in \mathbb N_0 \mbox{ for all } k\in \mathbb N_0    \Big\}.
\end{align}
\end{enumerate}
\end{definition}
\begin{remark}\label{rempsirho}
\begin{enumerate}[wide, labelindent=0pt]
\item For $x\in \mathcal P(\mathbb N_0)$ all of whose moments are finite we have
\begin{equation}
    \label{eq:psis2}
    \begin{aligned}
    \rho_n(x) & = \sum_{k=n}^\infty k(k-1)\cdots (k-n+1) x_k = (-1)^n \frac{\partial^n}{\partial s^n} \psi_s(x)\Big|_{s=0}, && n=1,2,... \\
      \psi_s(x) & = \sum_{k=0}^\infty (1-s)^k x_k = 1 + \sum_{n=1}^\infty \rho_n(x) \frac{(-s)^n}{n!}, && s\in [0,1], \\ 
    \end{aligned}
\end{equation}
where the last equality holds provided the series converges (which is certainly true for $x\in E_N$ or $x \in \mathcal{POI}$).  
\item An $x \in \mathcal P(\mathbb N_0)$ equals  $\rm{Poi}(\lambda)$ if and only if either of the following conditions (i) or (ii)  is satisfied:
$${\rm (i)} \quad \psi_s(x) = e^{-\lambda s}, \; s\in [0,1] \qquad {\rm (ii)} \quad  \rho_n(x) = \lambda^n, \; n=1,2,...$$
In particular we will make use of the fact that $\rho_2(x) - \rho_1^2(x) = 0$  for $x\in\mathcal{POI}$.
\end{enumerate}
\end{remark}
\subsection{Analysing the generator of \texorpdfstring{$X^N$}{XN}} \label{ss:32}
\noindent
We will now analyse the generator $G^N$ of $X^N$. Using the jump rates from \eqref{eq:jumpRate}, we find for  $f:E_N\to \mathbb R$, and $e_m$ the $m$th unit vector, 
\begin{align}\label{eq:GN}
  G^N f(x) & = \frac{N^2}{2} \sum_{m,n}  x_n \sum_{k,l} x_k x_l    \binom{k+l}{m} 2^{-(k+l)}\big(f(x + (e_m - e_n)/N) - f(x)\big).
\end{align}
First, we will analyse the action of the generator on functions only depending on $\rho_1(x)$ (Lemma~\ref{l1}), and on generating functions (Lemma~\ref{l2}). Afterwards, we are dealing with control of second and third moments, which we achieve by taking derivatives of the generating functions (Corollary~\ref{cor1}). In the next section, in Proposition \ref{P:3}, we will see that $|\rho_2(X^N_t) - \rho_1^2(X^N_t)|$ becomes small (cf. \eqref{eq:POIrho2}). Together with the following lemma, this points to the fact that $G^N$ acts on functions of the form $x\mapsto g(\rho_1(x))$ asymptotically like the generator of Feller's branching diffusion.

\begin{lemma}\label{l1}
  Let $G^N$ be as in \eqref{eq:GN} and $f$ be of the form $f(x) = g(\rho_1(x))$ for some $g\in\mathcal C_b^2(\mathbb R_+)$, thus \ $f$ only depends on the first moment of $x$. Then, 
  $$(G^N (g\circ \rho_1))(x) = \tfrac 12 \big(\rho_1(x) + \tfrac 34(\rho_2(x) - \rho_1^2(x))\big)g''(\rho_1(x)) + o(1) \quad \mbox{as }N\to \infty. $$
\end{lemma}


\begin{proof}
  Since $$g(\rho_1(x + (e_m-e_n)/N)) = g(\rho_1(x)) + N^{-1} g'(\rho_1(x))(m-n) + \tfrac 12 g''(\rho_1(x)) N^{-2}(m-n)^2 + o(N^{-2}),$$ we write
  \begin{align*}
      G^N & = NG_1 + G_0 + o(1), \\ 
      (G_1 (g\circ \rho_1))(x) & = \tfrac{1}{2}g'(\rho_1(x)) \sum_{m,n}  x_n \sum_{k,l} x_k x_l    \binom{k+l}{m} 2^{-(k+l)}(m-n)  \\ & = \tfrac 12g'(\rho_1(x))\Big( \sum_{k,l} x_k x_l \frac{k+l}{2}\underbrace{\sum_{m} \binom{k+l-1}{m-1}2^{-(k+l-1)}}_{=1} - \sum_n n x_n\Big) \\ & = \tfrac 12g'(\rho_1(x))\Big(\sum_k x_k \tfrac k2 + \sum_l x_l \tfrac l2 - \sum_n nx_n\Big)= 0, 
      \end{align*}
   \begin{align*}    
       (G_0 (g\circ \rho_1))(x) & =  \tfrac{1}{4}g''(\rho_1(x)) \sum_{m,n}  x_n \sum_{k,l} x_k x_l    \binom{k+l}{m} 2^{-(k+l)} \\ & \qquad \qquad \qquad \qquad \qquad \qquad \cdot (m(m-1) + n(n-1) -2mn + m + n)   \\ & = \tfrac{1}{4}g''(\rho_1(x))\Big( \sum_{k,l} x_k x_l \underbrace{\frac{(k+l)(k+l-1)}{4}}_{=(k(k-1)+l(l-1) + 2kl)/4} \\ & \qquad \qquad \qquad \qquad \qquad \qquad + \sum_n n(n-1)x_n - 2\rho^2_1(x) + 2\rho_1(x)\Big) \\ & = \tfrac 14 g''(\rho_1(x)) (\tfrac 12 \rho_2(x) + \tfrac 12 \rho_1^2(x) + \rho_2(x) - 2\rho_1^2(x) + 2\rho_1(x)) \\ & = \tfrac 12 g''(\rho_1(x)) (\rho_1(x) + \tfrac 34\rho_2(x) - \tfrac 34\rho_1^2(x)), 
  \end{align*}
  and the result follows. 
\end{proof}

\noindent
Now we analyse the structure of $G^N$ by collecting terms for the same powers in $N$, when using products of $\psi_s$.

\begin{lemma}\label{l2}
  \sloppy Let $G^N$ be as in \eqref{eq:GN}. For $\ell=1, 2,...$ and  $s:=(s_1,...,s_\ell) \in [0,1]^\ell$, set
  \begin{align}\label{eq:fs}
    f_s(x) & := \prod_{i=1}^\ell \psi_{s_i}(x).
  \end{align}
  Then there exist operators $G_1, G_0, G_{-1},\ldots$ obeying
  \begin{equation}\label{Grecursion}
  \begin{aligned}
    G^N f_s & = \sum_{i=1}^\ell N^{2-i} G_{2-i}f_s, 
    \\ 
    G_{2-i}f_s & = \tfrac 12\sum_{    \genfrac{}{}{0pt}{}{J\subseteq \{1,...,\ell\}}{|J| = i}} \Big(\prod_{j\notin J} \psi_{s_j}\Big) \sum_{K\subseteq J} (-1)^{|J\setminus K|} \psi^2_{(1-(1-s_K))/2} \psi_{1-(1-s_{J\setminus K})},
\end{aligned}
  \end{equation}
where the family of subsets $K\subseteq J$ includes also the empty set and $(1-s_K) :=\prod_{j\in K}(1-s_j)$ (and the empty product equals~1).
\end{lemma}

\begin{remark}[Form of $G_1$ and $G_0$]
Note that $G_{2-i}$ is of order $i$ in the sense that $G_{2-i}\psi_s$ is a sum over all possible subsets $J \subseteq \{1,...,\ell\}$ of cardinality $i$, where it leaves all factors $(\psi_{s_j})_{j \notin J}$ untouched and only acts on the factors $(\psi_{s_j})_{j \in J}$. So, in order to compute $G_1\psi_s$ and $G_0\psi_s$, it suffices to give these terms for $\ell=1$ and $\ell=2$, respectively. For $r,s \in [0,1]$, we find from \eqref{Grecursion}
\begin{equation}
  \label{eq:lfour3}
  \begin{aligned}
    2G_{1} \psi_s & = \psi_{s/2}^2 - \psi_s,\phantom{AAAAAA}
    \\ 2G_0 \psi_r \psi_s & = \psi_{1 - (1-r)(1-s)} - \psi_{r/2}^2\psi_s - \psi_{r}\psi_{s/2}^2 + \psi_{(1-(1-r)(1-s))/2}^2.
\end{aligned}
\end{equation}

\end{remark}

\begin{proof}
  We obtain, collecting terms proportional to $N^{2-i}$ for $i=1,...,\ell$, 
\begin{equation*}
    \label{eq:lfour2}
  \begin{aligned}
    G^Nf_s(x) & = \tfrac{N^2}{2}       \sum_{m,n}  x_n \sum_{k,\ell} x_k x_l \binom{k+\ell}{m} 2^{-(k+\ell)} \\ & \qquad \qquad \qquad \qquad \qquad \qquad \cdot \Big( \prod_{i=1}^\ell \Big(\underbrace{\psi_{s_i}(x) + N^{-1}((1-s_i)^m - (1-s_i)^n)\Big)}_{=\psi_{s_i}(x + (e_m - e_n)/N)} - f_s(x)\Big) \\ & = \tfrac 12 \sum_{m,n}  x_n \sum_{k,\ell} x_k x_\ell    \binom{k+\ell}{m} 2^{-(k+\ell)} \sum_{i=1}^\ell N^{2-i} \\ & \qquad \qquad \qquad \qquad \qquad \qquad \sum_{\genfrac{}{}{0pt}{}{J\subseteq \{1,...,\ell\}}{|J| = i}} \prod_{j\in J}((1-s_j)^m - (1-s_j)^n) \prod_{j\notin J} \psi_{s_j}(x) \\ & = \sum_{i=1}^\ell N^{2-i} G_{2-i} f_s(x)
    \end{aligned}
    \end{equation*}
    with
    \begin{equation}
    \label{eq:lfour2b}
  \begin{aligned}
    2G_{2-i}f_s(x) & = \!\!\!\!\!\! \sum_{\genfrac{}{}{0pt}{}{J\subseteq \{1,...,\ell\}}{|J| = i}}\!\!\! \Big(\prod_{j\notin J} \psi_{s_j}(x)\Big) \sum_{m,n}  x_n \sum_{k,\ell} x_k x_\ell    \binom{k+\ell}{m} 2^{-(k+\ell)} \prod_{j \in J} ((1-s_j)^m -  (1-s_j)^n) \\ & = \!\!\!\!\!\! \sum_{\genfrac{}{}{0pt}{}{J\subseteq \{1,...,\ell\}}{|J| = i}} \Big(\prod_{j\notin J} \psi_{s_j}(x)\Big) \sum_{K\subseteq J} (-1)^{|J\setminus K|} \sum_{k,\ell,m} x_k x_\ell    \binom{k+l}{m} 2^{-(k+\ell)}  \Big(\prod_{j \in K} (1-s_j)\Big)^m  \\ & \qquad \qquad \qquad \qquad \qquad \qquad \qquad \qquad \qquad \qquad \cdot \sum_{n}  x_n \Big(\prod_{j \in J\setminus K} (1-s_j)\Big)^n.
    \end{aligned}
\end{equation}
Taking the sum over $m$, the right hand side of \eqref{eq:lfour2b} turns into \begin{equation*}
\sum_{\genfrac{}{}{0pt}{}{J\subseteq \{1,...,\ell\}}{|J| = i}} \Big(\prod_{j\notin J} \psi_{s_j}(x)\Big) \sum_{K\subseteq J} (-1)^{|J\setminus K|} \sum_{k,l} x_k x_\ell \Big(\frac{1 + \prod_{j \in K} (1-s_j)}{2} \Big)^{k+\ell} 
\psi_{1 - \prod_{j \in J\setminus K}(1-s_j)}(x)
\end{equation*}
\vspace{-0.9cm}
\begin{equation*}
\begin{aligned}
 \\ & =  \sum_{\genfrac{}{}{0pt}{}{J\subseteq \{1,...,\ell\}}{|J| = i}} \Big(\prod_{j\notin J} \psi_{s_j}(x)\Big) \sum_{K\subseteq J} (-1)^{|J\setminus K|} \psi_{(1 -  (1-s_K))/2}^2(x) \psi_{1 - (1-s_{J\setminus K})}(x).
\end{aligned}
\end{equation*}
We have thus obtained \eqref{Grecursion}. From the form of $G_{2-i}$ we see that $G_{2-i} f_s = 0$ if $i> \ell$ (since there is no $J \subseteq \{1,...,\ell\}$ with $|J| = i$ and the outer sum is empty). 
\end{proof}
\begin{corollary}[Martingale property of $Z^N$ and uniform bounds on 2nd and 3rd moments]\label{cor1}
a) The process $Z^N = \rho_1(X^N_\cdot)$ is a martingale. \\
b) For all $t$ and $N$, and for some constants $C,C'<\infty$, which do not depend on $N$,
  \begin{align}\label{secondmom}
  \mathbb E[\rho_2(X^N_t)] & \leq   \mathbb E[\rho_1(X^N_0)] (t+1) +  \mathbb E[\rho_2(X^N_0)],
  \\ \label{thirdmom}
  \mathbb E[\rho_3(X^N_t)] & \leq C t + C'\mathbb E[\rho_3(X^N_0)].
  \end{align}
\end{corollary}

\begin{proof} Using \eqref{eq:psis2} and \eqref{eq:lfour3} we obtain
\begin{align} \label{rho1}
-G^N \rho_1(x) & = \frac d{ds} G^N \psi_s(x) |_{s=0} = \frac d{ds} \frac N2(\psi_{s/2}^2(x)  - \psi_s(x))|_{s=0} = 0, \\ \label{rho1sqare}
  G^N \rho_1^2(x) & = \rho_1(x) + \tfrac 34 (\rho_2(x) - \rho_1^2(x)), \\
\label{rhotwo}
    G^N \rho_2(x) & = \frac{d^2}{ds^2} G^N \psi_s(x)\Big|_{s=0} = \frac N2 \frac{d^2}{ds^2} (\psi_{s/2}^2(x) - \psi_s(x))\Big|_{s=0} \\ 
    \notag & = \frac N2 (\tfrac 12 \rho_1^2(x) + \tfrac 12 \rho_2(x) - \rho_2(x)) = \tfrac N4(\rho_1^2(x) - \rho_2(x)),  
      \end{align}
  Assertion a) is immediate from \eqref{rho1}. Combining  \eqref{rho1sqare} and \eqref{rhotwo} we see that
  \begin{align}\label{ODE}
      \frac d{dt} \mathbb E[\rho_1^2(X_t^N) - \rho_2(X_t^N)] & = -\tfrac{N+3}{4} \mathbb E[\rho_1^2(X_t^N) - \rho_2(X_t^N)] + \mathbb E[\rho_1(X_t^N)].
  \end{align}
  Since $(\rho_1(X_t^N))_{t\geq 0}$ is a martingale, we have $\mathbb E[\rho_1(X_t^N)]\equiv z$, and we can solve \eqref{ODE} by
  \begin{align*}
    \mathbb E[\rho_1^2(X_t^N) - \rho_2(X_t^N)] & = e^{-(N+3)t/4}\mathbb E[\rho_1^2(X^N_0) - \rho_2(X^N_0)] + (1 - e^{-(N+3)t/4}) \tfrac{4}{N+3}z.   \end{align*}
  As a consequence, 
  \begin{align*}
   \mathbb E[\rho_2 & (X_t^N)] - \mathbb E[\rho_2(X_0^N)] = \tfrac{N}{4}\int_0^t  \mathbb E[\rho_1^2(X_s^N) - \rho_2(X_s^N)]) ds
   \\ & =  \tfrac{N}{N+3}(1 - e^{-(N+3)t/4}) \mathbb E[\rho_1^2(X_0^N) - \rho_2(X_0^N)] + \tfrac{N}{N+3} z (t - \tfrac{4}{N+3}(1 - e^{-(N+3)t/4}))
   \\ & \leq \mathbb E[\rho_1^2(X_0^N) - \rho_2(X_0^N)] + zt,
  \end{align*}
  and \eqref{secondmom} follows from $\rho_1^2(X_0^N) - \rho_2(X_0^N) \leq \rho_1(X_0^N)$. For \eqref{thirdmom}, we employ  similar calculations as in \eqref{rhotwo} and obtain \begin{align*}
      G^N \rho_3 & = \tfrac{3N}{8}(\rho_2\rho_1 - \rho_3), \\       
      G^N (\rho_2\rho_1) & = \tfrac N4 \rho_1(\rho_1^2 - \rho_2) + \tfrac 32\rho_2 + \tfrac 12\rho_1^2 + \tfrac 58 \rho_3 - \tfrac 58 \rho_2 \rho_1, \\ 
      G^N\rho^3_1 & = \tfrac 94 \rho_2\rho_1 \!- \!\tfrac 94 \rho_1^3 + 9\rho_1^2 \! +\! \tfrac{1}{N}\left( - \tfrac 38 \rho_3 + \tfrac 98 \rho_2\rho_1 - \tfrac 34 \rho_1^3 + \tfrac 38 \rho_1^2 - \tfrac 98 \rho_2\right).
  \end{align*}
 Thus for $a, b, c \in \mathbb R$ we have
  \begin{align*}
      G^N & (a\rho_3 + b \rho_2 \rho_1 + c \rho_1^3) \\ &  = (- a\tfrac{3N}{8} + b\tfrac{5}{8} - c \tfrac 3{8N})\rho_3 + \left( a\tfrac{3N}{8} -  b (\tfrac N4 + \tfrac{5}{8}) + c (\tfrac 94 + \tfrac{9}{8N})\right)\rho_2\rho_1 \\ & \qquad \qquad \qquad \qquad + \left(b \tfrac N4 - c (\tfrac 94 + \tfrac{3}{4N})\right) \rho_1^3 \\ &  \qquad \qquad \qquad \qquad + \left(b \tfrac 32 - c \tfrac 9{8N}\right) \rho_2 + \left(b \tfrac 12 + c (9 + \tfrac{3}{8N} )\right)\rho_1^2.
  \end{align*}
From this we deduce the existence of numbers 
  $a_1 = 1, b_1 = -3 + o(1), c_1 = 2 + o(1)$, $a_2 = o(1), b_2 = \tfrac 15 + o(1), c_2 = -\tfrac 15 + o(1)$ and
$\lambda_1 = \tfrac 38 + o(1), \lambda_2 = \tfrac 14$ such that the function
$f^N_i(x) := a_i \rho_3(x) + b_i \rho_2(x)\rho_1(x) + c_i \rho_1^3(x)$ obeys
\begin{align*}
  G^N f^N_i(x) & = -\lambda_i N f^N_i(x) 
  + \left(b_i \tfrac 32 - c_i \tfrac 9{8N}\right) \rho_2(x) + \left(b_i \tfrac 12 + c_i (9 + \tfrac{3}{8N} )\right)\rho_1^2(x), \qquad i=1,2.
\end{align*}
(Here, we do not display the exact form of the six $o(1)$ terms,
which we found using the  computer algebra system {\tt sagemath} \citep{sagemath}. The sagemath commands are given in the latex sourcefile of the arxiv version of our paper.)
In particular, this shows that $G^Nf_i^N(x)+\lambda_i Nf_i^N(x)$ is a bounded function of $(\rho_2(x), \rho_1^2(x)$. Thus, with the same calculation as for the second moments we obtain
\begin{align}\label{eq:dsa}
\mathbb E[f^N_i(X_t^N)] & = e^{-\lambda_i N t} \mathbb E[f^N_i(X_0^N)] + O(N^{-1}), \qquad i=1,2.    
\end{align}
Finally, we find $y = 1 + o(1)$ and $z = 10 + o(1)$ with
$yf_1^N(x) + zf^N_2(x) = \rho_3(x) - \rho_2(x)\rho_1(x)$ and therefore,
\begin{align*}
  \mathbb E[\rho_3(X_t^N)] - \mathbb E[\rho_3(X_0^N)] & = -\tfrac{3N}{8} \int_0^t \mathbb E[y f_1(X_s^N) + z f_2(X_s^N)] ds. 
\end{align*}
Plugging in \eqref{eq:dsa}, we obtain the assertion \eqref{thirdmom}.
\end{proof}

\subsection{Random Poisson distributions}\label{ss:33}
Recall from \eqref{occmeas} the occupation measure $\Gamma_{X^N}$ of $X^N$ (which is a random probability measure on $\mathbb R_+ \times \mathcal P(\mathbb N_0)$). We will show in this section:

\begin{proposition}\label{P:3}
  Let the assumptions from Theorem~\ref{T1} be satisfied. Then \\
  (a) the sequence $(\Gamma_{X^N})_{N=1,2,...}$ is tight,
  \\ (b) any limit point $\Gamma$ is concentrated on the set of Poisson distributions, in the sense that  $\Gamma([0,\infty) \times \mathcal{POI}) = 1$, and
  \begin{align}\label{eq:POIrho2}
    \int_0^\infty e^{-s} \mathbb E[|\rho_2(X_s^N) - \rho_1^2(X_s^N)|] ds = \mathbb E\Big[\int |\rho_2(x) - \rho_1^2(x)|\Gamma_{X^N}(ds, dx)\Big] \xrightarrow{N\to\infty} 0.
  \end{align}
  \end{proposition}
  
  \noindent
We prepare the proof of this proposition by  a characterization of random Poisson distributions on $\mathbb N_0$.

\begin{lemma}[Characterization of random Poisson distributions]\label{l:4}
  Let $\psi_s$ and $\rho_n$ be as in Definition~\ref{def:1}. Let $Y$ be a $\mathcal P(\mathbb N_0)$-valued random variable with $\rho_1(Y) < \infty$ almost surely. Then the following are equivalent:
  \begin{enumerate}[wide, labelindent=0pt]
      \item $Y$ has the same distribution as $\rm{Poi}(\rho_1(Y))$.
      \item $\mathbb E[e^{-\lambda \psi_s(Y)}|\rho_1(Y)] = \exp\left({-\lambda e^{-s\rho_1(Y)}}\right)$ almost surely, for all $\lambda\geq 0$.
\item For all $n=1,2,...$ and  $s_1,...,s_n\in[0,1]$, we have almost surely $$\mathbb E[\psi_{s_1}(Y) \cdots \psi_{s_n}(Y)|\rho_1(Y)] = \frac 1n \sum_{j=1}^n \mathbb E\Big[\psi_{s_j/2}^2(Y) \prod_{\genfrac{}{}{0pt}{}{k=1}{k\neq j}}^n \psi_{s_{k}}(Y)|\rho_1(Y)\Big].$$
\end{enumerate}
\end{lemma}

\begin{proof}
    \noindent $1.\Rightarrow 3.:$ By assumption we have almost surely
   \begin{align*}
   \mathbb E[\psi_{s_1}(Y) \cdots \psi_{s_n}(Y)|\rho_1(Y)] & = \mathbb E[e^{-(s_1 + \cdots + s_n)\rho_1(Y)}|\rho_1(Y)] = e^{-(s_1 + \cdots + s_n)\rho_1(Y)}.
  \end{align*}
  Since the right hand side only depends on $s_1 + \cdots + s_n$, the result follows from taking expectations and summing in $$\mathbb E[\psi_{s_1}(Y) \cdots \psi_{s_n}(Y)|\rho_1(Y)] = \mathbb E\Big[\psi_{s_j/2}^2(Y) \prod_{\genfrac{}{}{0pt}{}{k=1}{k\neq j}}^n \psi_{s_{k}}(Y)|\rho_1(Y) \Big].$$ 
  
  \noindent $3.\Rightarrow 2.:$ We start with the following observation: For $s>0$ let $(s_{kj})_{k\in \mathbb N, j=1,...,k}$ be asymptotically negligible (in the sense that\ $\sup_j |s_{kj}| \xrightarrow{k\to\infty} 0$) and $\sum_{j=1}^k s_{kj} = s$. Then, since $$\psi_{s_{kj}}(Y) = \sum_{i=0}^\infty Y_i (1-s_{kj})^i = 1 - (s_{kj} + o(s_{kj}))\sum_i iY_i,$$ we have
  \begin{align}\label{eq:911}
    \log\Big(\prod_{j=1}^k \psi_{s_{kj}}(Y)\Big) & = \sum_{j=1}^k \log(1 - (s_{kj} + o(s_{kj}))\rho_1(Y)) \xrightarrow{k\to\infty}  -s\rho_1(Y).
  \end{align}
  \sloppy Now, we come to proving the assertion: Fix $s\in [0,1]$ and $n=1,2,...$, and let $\Pi_k$ be a random partition of $[0,ns]$ with $k$ elements, which arises iteratively as follows: Starting with $\Pi_n = \{[0,s), [s, 2s),...,[(n-1)s, ns)\}$, let $\Pi_{k+1}$ arise from $\Pi_k$ by randomly taking one partition element $[a,b]$ from $\Pi_k$, and adding the two elements $[a,(a+b)/2)$ and $[(a+b)/2, b)$ to $\Pi_{k+1}$. (For $n=1$, we can e.g.\ have $\Pi_1 = \{[0,s)\}, \Pi_2 = \{[0,s/2), [s/2,s)\}, \Pi_3 = \{[0,s/4), [s/4, s/2), [s/2, s)\}, \Pi_4 = \{[0,s/4), [s/4, 3s/8), [3s/8, s/2), [s/2, s)\},...)$. From 3. we  find almost surely
  \begin{align*}
    \mathbb E[\psi_s^n(Y)|\rho_1(Y)] & = \mathbb E\Big[\prod_{\pi \in \Pi_n} \psi_{|\pi|}(Y)|\rho_1(Y)\Big] = \mathbb E\Big[\prod_{\pi \in \Pi_{n+1}} \psi_{|\pi|}(Y)|\rho_1(Y)\Big],  
  \end{align*}
  since the expectation of the right hand side also runs over the $n$ possible random partitions $\Pi_{n+1}$, which arise by splitting one element of length $s$ into two elements of length $s/2$. It is not hard to see that -- almost surely -- every partition element in $\Pi_k$ eventually gets split in two, so $\{|\pi|: \pi \in \Pi_k\}$ is asymptotically negligible as $k\to\infty$. Therefore, $\prod_{\pi \in \Pi_k} \psi_{|\pi|}(Y) \xrightarrow{k\to\infty} e^{-ns\rho_1(Y)}$ almost surely as in \eqref{eq:911}, so we see using dominated convergence that almost surely
  $$ \mathbb E[\psi_s^n(Y)|\rho_1(Y)] = \lim_{k\to\infty} \mathbb E\Big[\prod_{\pi\in\Pi_k} \psi_{|\pi|}(Y)|\rho_1(Y)\Big] = e^{-ns\rho_1(Y)}.$$ Therefore, 
  $$ \mathbb E[e^{-\lambda \psi_s(Y)}|\rho_1(Y)] = \sum_{n=0}^\infty \frac {(-\lambda)^n}{n!}  e^{-ns\rho_1(Y)} = e^{-\lambda e^{-s\rho_1(Y)}}.$$
  2.$\Rightarrow$1.: Conditional on $\rho_1(Y)$, we see that almost surely $\psi_s(Y) = e^{-s\rho_1(Y)}$ for all $s\in [0,1]$, which means that $Y \sim \text{Poi}(\rho_1(Y))$. 
\end{proof}

\begin{proof}[Proof of Proposition \ref{P:3}] \sloppy
%
   The occupation measures $\Gamma_{X^N}$, defined according to \eqref{occmeas}, are random elements of $\mathcal P([0,\infty) \times \mathcal P(\mathbb N_0))$. By Markov's inequality and Prohorov's criterion, for $C >0$  the set  $$K_{C}:=  \{x\in \mathcal P(\mathbb N_0): \rho_1(x) \le C\}$$
   is relatively compact in $\mathcal P(\mathbb N_0)$.
   Hence, again by Prohorov's criterion, tightness of $(\Gamma_{X^N})$ follows if for all $\varepsilon > 0$ there exist $T, C<\infty$    such that 
   \begin{equation} \label{tight} \mathbb P(\Gamma_{X^N}([0,T] \times K_{C}) \ge 1-\varepsilon) \ge 1-\varepsilon\,\, \mbox{ for all } N.
   \end{equation}
   For given $\varepsilon > 0$ let $T$ be such that $\int_0^T e^{-t} dt \ge 1-\varepsilon$.
Because of Corollary \ref{cor1}, $\rho_1(X^N_\cdot)$ is a non-negative martingale, hence  by Doob's inequality and due to the assumptions of Theorem \ref{T1}  there exists a  finite constant $C$ such that 
\begin{align}\label{Doob}
\mathbb P(\sup_{0\le t \le T} \rho_1(X^N_t) \le C) \ge 1- \varepsilon.
\end{align}
The event in the l.h.s.~of \eqref{Doob} equals $\{X^N_t \in  K_C \mbox{ for all } t \in [0,T]\}$, and due to our assumption on $T$ this implies the event in the l.h.s of \eqref{tight}. Thus we infer the validity of \eqref{tight}, showing assertion (a) of the proposition.
   
 In order to prove assertion (b), recall $f_s$ from \eqref{eq:fs} and choose $g$ bounded and smooth. We are going to argue using the martingale
    \begin{align}\label{eq:mart}
    \Big(N^{-1} & g(\rho_1(X^N_t))f_s(X^N_t) - \int_0^t N^{-1} G^N g(\rho_1(X^N_r)) f_s(X^N_r)dr\Big)_{t\geq 0}. 
    \end{align}
    As $N\to\infty$, we find that $N^{-1} g(\rho_1(X^N_t))f_s(X^N_t) \xrightarrow{N\to\infty} 0$ since $g$ and $f_s$ are bounded. Moreover, from Lemma~\ref{l1} and Lemma~\ref{l2},
    \begin{align*}
      N^{-1}G^N g(\rho_1(x))f_s(x) & = g(\rho_1(x)) G_1 f_s(x) + O(N^{-1}), 
    \end{align*}
    and $G_1f_s$ is given by \eqref{eq:lfour3}. Therefore, multiplying \eqref{eq:mart} by $2/\ell$ and taking $N\to\infty$, we find that for any weak limit $\Gamma$ of $\Gamma_{X^N}$ as $N\to\infty$ that
    \begin{align*}
         \tfrac 1 \ell \Big( \int_0^t e^r  \Big(g(\rho_1(x))\sum_{j=1}^\ell \Big(\psi_{s_j/2}^2(x) \prod_{\genfrac{}{}{0pt}{}{k=1}{k\neq \ell}}^\ell \psi_{s_k(x)} - \prod_{k=1}^\ell \psi_{s_k(x)}\Big) \Big) \Gamma(dr, dx)\Big)_{t\geq 0}
    \end{align*}
    is a martingale with continuous paths and vanishing quadratic variation, hence vanishes. Consequently, if $Y$ is a $\mathcal P(\mathbb N_0)$-valued random variable which,  given $\Gamma$, has distribution  $\bar \Gamma(dx) := \Gamma([0,\infty) \times dx)$, we have
   \begin{align*} \frac 1 \ell \mathbb E \Big[ g(\rho_1(Y))\sum_{j=1}^\ell \psi_{s_j/2}^2(Y) \prod_{\genfrac{}{}{0pt}{}{k=1}{k\neq \ell}}^\ell \psi_{s_k(Y)}\Big] = \mathbb E\Big[ g(\rho_1(Y)) \prod_{k=1}^\ell \psi_{s_k(Y)} \Big],
   \end{align*}
   \vspace{-0.2cm}
    which implies
     \begin{align*} \frac 1 \ell \mathbb E \Big[ \Big(\sum_{j=1}^\ell \psi_{s_j/2}^2(Y) \prod_{\genfrac{}{}{0pt}{}{k=1}{k\neq \ell}}^\ell \psi_{s_k(Y)}| \rho_1(Y) \Big] = \mathbb E\Big[ \prod_{k=1}^\ell \psi_{s_k(Y)} | \rho_1(Y) \Big]
     \end{align*}
    since $g$ was arbitrary. From Lemma~\ref{l:4}, we see that $\bar\Gamma$ must be concentrated on $\mathcal{POI}$. \\Finally, for proving \eqref{eq:POIrho2}, consider a probability space with $Y^N \sim \bar \Gamma_{X^N}$, $Y\sim \bar\Gamma_X$ and \mbox{$Y^N \xrightarrow{N\to\infty} Y$} almost surely.
  From Corollary~\ref{cor1} we see, using Jensen's inequality,  that   
  $$ \sup_{N}\mathbb E[(\rho_2(Y^N)^{3/2}] = \sup_{N}\int_0^\infty e^{-t} \mathbb E[(\rho_2(X^N_t)^{3/2}]dt \leq \sup_{N} \int_0^\infty e^{-t} \mathbb E\Big[ \sum_{k=2}^\infty k^3 X_t^N(k)\Big]dt < \infty,$$ 
  which implies that  $(\rho_2(Y^N))_N$ and $(\rho_1^2(Y^N))_N$ are uniformly integrable. Hence, 
  $$ \int_0^\infty e^{-t} |\rho_2(X^N_t) - \rho_1^2(X^N_t)| dt = |\rho_2(Y^N) - \rho_1^2(Y^N)| \xrightarrow{N\to\infty} |\rho_2(Y) - \rho_1^2(Y)| = 0$$ in $L^1$, and \eqref{eq:POIrho2} follows. 
\end{proof}

\subsection{Convergence of $Z^N$ to Feller's branching diffusion}\label{ss:34}
%
\begin{proposition}\label{P:2}
  Let the assumptions of Theorem~\ref{T1} be satisfied. Then $Z^N$ converges as $N\to \infty$ in distribution with respect to the Skorokhod topology on $\mathcal D(\mathbb R_+)$ to $Z$, where $Z$ solves  \eqref{FBD}.  
\end{proposition}

\begin{proof}
  For the claimed limit of $Z^N=\rho_1(X^N_\cdot)$ as $N\to \infty$, we need to show existence (i.e.\ tightness) and uniqueness.
  Recall from Corollary \ref{cor1} and eq. \eqref{rho1sqare} that $(\rho_1(X^N_\cdot))_{N=1,2,...}$  are non-negative martingales and have quadratic variation processes $M^1, M^2,...$ with
  $$ M^N_t = \int_0^t \rho_1(X^N_s) + \tfrac 32\big(\rho_2(X^N_s) - \rho_1^2(X^N_s)\big) ds.$$
  For tightness of $(\rho_1(X^N_\cdot))_{N=1,2,...}$, we use the Aldous-Rebolledo criterion, see e.g. Theorem~1.17 in \cite{MR1779100}. So, we have to show that
  \begin{enumerate}
      \item for all $t\geq 0$, the family $(\rho_1(X^N_t))_{N=1,2,...}$ is tight;
      \item for every sequence $\tau_1, \tau_2,...$ of stopping times, bounded by $T<\infty$ and for every $\varepsilon > 0$, there exists $\delta > 0$ such that
      $$ \limsup_{N\to\infty} \sup_{0\leq \theta\leq \delta} \mathbb P(M^N(\tau_N + \theta) - M^N(\tau_N) > \varepsilon) < \varepsilon.$$
  \end{enumerate}
  For 1., the Markov inequality implies that for $\varepsilon>0$ there exists a finite constant $C_\varepsilon$ independent of $t$ and $N$ such that
  $$ \sup_{N}\mathbb P(\rho_1(X^N_t) > C_\varepsilon) \leq \varepsilon.$$
  In particular, this implies that $(\rho_1(X^N_t))_{N=1,2,...}$ is tight for all $t\geq 0$. 
  For 2., we take  $\varepsilon, \tau_1, \tau_2,...$ , as above, and use  $\delta := \varepsilon / (2C_\varepsilon)$ and write
  \begin{align*}
    \limsup_{N\to\infty}       & \sup_{0\leq \theta \leq \delta} \mathbb P  \Big( \int_{\tau_N}^{\tau_N + \theta} \rho_1(X^N_s) + \tfrac 32\big(\rho_2(X^N_s) - \rho_1^2(X^N_s)\big) ds > \varepsilon\Big) \\ & \leq \limsup_{N\to\infty} \mathbb P\Big( \int_{\tau_N}^{\tau_N + \delta} \rho_1(X^N_s) ds > \varepsilon/2\Big) \\ & \qquad \qquad \qquad \qquad + \mathbb P\Big[ \int_0^T \big|\rho_2(X^N_s) - \rho_1^2(X^N_s)\big| ds > \varepsilon/2\Big] \\ & \leq \sup_{N} \mathbb P(\sup_{0\leq t\leq T} \rho_1(X^N_t) >  \varepsilon / (2\delta)) \leq \varepsilon,
  \end{align*}
  where we have used Proposition~\ref{P:3} in the second to last, and the martingale property of $\rho_1(X_\cdot^N)$ (see Corollary~\ref{cor1}) together with Doob's maximal inequality in the last step.\\
  For uniqueness, we use Theorem~8.2.b $\Rightarrow$ a in \cite{EthierKurtz1986}, and the fact that the generator of $Z$ is $Gf(z) = \tfrac 12 z f''(z)$ for $f\in\mathcal C_b^2(\mathbb R_+)$. We write
  \begin{align*}
      f(\rho_1(X^N_{t+s})) & - f(\rho_1(X^N_t)) - \int_t^{t+s} Gf(\rho_1(X^N_u)) du \\ & = f(\rho_1(X^N_{t+s})) - f(\rho_1(X^N_t)) - \int_t^{t+s} G^Nf(\rho_1(X^N_u)) du \\ & \qquad + \tfrac 38 \int_t^{t+s} f''(\rho_1(X^N_u))(\rho_2(X^N_u) - \rho_1^2(X^N_u))
  \end{align*}
  by Lemma~\ref{l1}. Hence, (8.7) in \cite{EthierKurtz1986} follows from Proposition~\ref{P:3}.
\end{proof}

\subsection{Completion of the proof of Theorem~\ref{T1}}\label{ss:35}
We know from Proposition \ref{P:3} that the sequence $(\Gamma_{X^N})$ is tight. In addition, noting that the topology of convergence in measure on $\mathcal D(\mathbb R_+)$ is weaker than the Skorokhod topology, we see from Proposition~\ref{P:3} that $\Gamma_{Z^N} \xRightarrow{N\to\infty} \Gamma_Z$. Hence, we see that the family of bi-variate random probability measures $(\Gamma_{(X^N, Z^N})$ is tight, and  
$$ \Gamma_{(X^N, Z^N)}(A) = 1, \qquad \text{with} \qquad A := [0,\infty) \times  \big(\{(x,z) \in \mathcal P(\mathbb N_0) \times \mathbb R_+: \rho_1(x) = z\}\big).$$
Let us define for $C > 0$ the set 
$$B_{C}:= [0,\infty)\times \{x \in \mathcal P(\mathbb N_0): \rho_2(x) \le C\} \times [0,\infty),$$ 
and note that $A\cap B_{C}$ is closed (since $x\mapsto \rho_1(x)$ is continuous on a set with bounded second moments). By Corollary \ref{cor1}, for all   $\varepsilon>0$ there is $C>0$ such that
for all~$N$
\begin{align}
    \label{eq:lk}
    \mathbb P(  \Gamma_{(X^N, Z^N)} (B_{C}) \ge 1-\varepsilon) \ge 1-\varepsilon.
\end{align}
Let us consider a weak limit $\Gamma$ of $\Gamma_{(X^{N'}, Z^{N'})}$ along a subsequence $N'$. We may choose a probability space where $\Gamma_{(X^{N'}, Z^{N'})} \xRightarrow{N'\to\infty} \Gamma$ almost surely. On this space, by the Portmanteau Theorem, 
\begin{align*}
    \Gamma(A) & \geq  \Gamma(A \cap B_{C}) \geq  \limsup_{N'\to\infty} \Gamma_{(X^{N'}, Z^{N'})}(A \cap B_{C}) = \limsup_{N'\to\infty} \Gamma_{(X^{N'}, Z^{N'})}(B_{C}). 
\end{align*}
For every $\varepsilon>0$, choosing $C>0$ such that \eqref{eq:lk} holds, we therefore find
$$ \mathbb P(\Gamma(A)\geq 1-\varepsilon) \geq 1-\varepsilon$$
which implies $\Gamma(A) = 1$ almost surely. From Proposition \ref{P:3} we know that 
$$ \Gamma(D) = 1, \qquad \text{ with } \qquad D = [0,\infty) \times \mathcal{POI} \times [0,\infty)$$
almost surely, and we can conclude that $\Gamma (A\cap D) = 1$ almost surely. Hence, we are done by noting that
$$ A\cap D = [0,\infty) \times \{(x,z): x = \text{Poi}(z)\}.$$

\subsubsection*{Acknowledgements} We thank Peter Czuppon, Jonathan Henshaw and Judith Korb for discussions concerning the dynamics of transposable elements. PP is partly supported by the Freiburg Center for data analysis and modeling (FDM). AW is partly supported by DFG SPP 1590.


\newpage

\noindent

\end{document}